\documentclass[12pt,a4paper]{article}
\usepackage[utf8]{inputenc}
\usepackage[english]{babel}
\usepackage{amsmath}
\usepackage{amsfonts}
\usepackage{amssymb}
\usepackage{amsthm}
\usepackage{graphicx}
\usepackage{todonotes}
\usepackage[left=2.5cm,right=2.5cm,top=2.3cm,bottom=3.8cm]{geometry}

\newtheorem{theorem}{Theorem}

\newtheorem{lemma}[theorem]{Lemma}

\theoremstyle{definition}
\newtheorem{remark}[theorem]{Remark}

\newcommand{\D}{2(m+1) \cdot \exp\left(13m X(m)\right) \cdot 4^m}

\begin{document}

\author{Davide Lombardo\footnote{Département de Mathématiques d'Orsay}}
\date{}
\title{On the analytic bijections of the rationals in $[0,1]$}

\maketitle

\begin{abstract}
We carry out an arithmetical study of analytic functions $f: [0,1] \to [0,1]$ that by restriction induce a bijection $\mathbb{Q} \cap [0,1] \to \mathbb{Q} \cap [0,1]$. The existence of such functions shows that, unless $f(x)$ has some additional property of an algebraic nature, very little can be said about the distribution of rational points on its graph. Some more refined questions involving heights are also explored.
\end{abstract}

\section{Introduction}
In a recent conversation Umberto Zannier, with an eye to arithmetical applications, asked whether there exist analytic functions $f:[0,1] \to [0,1]$ that induce bijections of $\mathbb{Q} \cap [0,1]$ with itself and that are not rational functions.
This is indeed an interesting question, because it helps shed light on the kind of hypotheses necessary on a function $f(x)$ in order to study the distribution of rational points on its graph, and as it turns out the answer is affirmative: transcendental, analytic functions that induce bijections of $\mathbb{Q}\cap[0,1]$ with itself do exist. In particular, if $g(x)$ is a general analytic function, satisfying no particular algebraic property, then very little information on the distribution of rational points on the graph of $g$ can be obtained besides that afforded by the theorems of Bombieri-Pila \cite{MR1016893}, Pila \cite{MR1115117}, and Pila-Wilkie \cite{MR2228464}, which was the original motivation of Zannier's question.
One should contrast this fact with the much tamer behaviour exhibited by algebraic functions:
\begin{lemma}\label{lemma:OnlyRationalFunctions}
Suppose $f:[0,1] \to [0,1]$ is algebraic and induces a bijection of $\mathbb{Q} \cap [0,1]$ with itself: then $f(x)$ is a linear fractional transformation (that is, a rational function of degree one) with rational coefficients. More precisely, there exists $a \in \mathbb{Q}$ such that either $f(x)=\frac{x}{ax+(1-a)x}$ or $f(x)=\frac{(a-1)(x-1)}{ax+(1-a)}$.
\end{lemma}

\begin{proof}
Since $f(x)$ is algebraic, there exists a polynomial $p(x,y) \in \mathbb{Q}[x,y]$ such that $p(x,f(x))$ is identically zero. Suppose first that $\deg_y p(x,y) \geq 2$. By Hilbert's irreducibility theorem, we can find a rational number $x_0 \in [0,1]$ such that $p(x_0,y) \in \mathbb{Q}[y]$ is irreducible of degree $\geq 2$: but this implies that $f(x_0)$, which by definition is a root of the equation $p(x_0,y)=0$, is not a rational number, contradiction. Conversely, suppose that $\deg_xp(x,y) \geq 2$. Then by Hilbert irreducibility again there exists $y_0 \in \mathbb{Q} \cap [0,1]$ such that $p(x,y_0)$ is irreducible of degree at least 2: but this implies that the inverse image of $y_0$ via $f$ is not rational, contradiction. So $p(x,y)$ is linear in $x$ and $y$, hence $f(x)$ is a linear fractional transformation. One checks easily that the only linear fractional transformations that induce bijections of $\mathbb{Q} \cap [0, 1]$ are those given in the statement.
\end{proof}

Notice that this lemma -- which is well-known to experts -- gives an easy criterion to show that the functions we construct are transcendental (see for example remark \ref{rmk:DifferentProof}).
While investigating Zannier's question, I found out that the existence of functions as those studied in this paper had already been established by Franklin \cite{MR1501300}, but his construction was somewhat indirect and his point of view mostly analytical, which makes his approach not especially well-suited to study arithmetical questions. 

In this note, on the other hand, we consider the problem from a more arithmetical standpoint: in particular, we give a new, slightly simplified construction (section \ref{sec:Constr}) which, being more explicit than Franklin's, also allows us to treat problems in the spirit of the Bombieri-Pila, Pila, and Pila-Wilkie counting theorems. We show for example (section \ref{sec:HeightBounds}) that the functions produced from a further refining of our construction satisfy an inequality of the form $h(f(x)) \leq b(h(x))$ for all $x \in \mathbb{Q} \cap [0,1]$, where by $h(x)$ we mean the standard logarithmic height of the rational number $x$ and $b(t)$ is a certain explicit bound function. We also prove (section \ref{sec:ManyRationalPoints}) that the graph of these bijections $f(x)$ can be made to contain ``many'' rational points of bounded height, in the sense of the Pila counting theorem. Finally, our explicit descriptions also make it clear that, unlike what happens with -- say -- rational functions, for the functions $f(x)$ we construct there are infinitely many rational numbers in $[0,1]$ for which the height of $f(x)$ is dramatically smaller than the height of $x$. It is this last phenomenon in particular that makes it impossible to gain more information on the distribution of rationals points on the graph of $f(x)$ besides what is already contained in the theorems of Bombieri, Pila, and Wilkie.

\section{The basic construction}\label{sec:Constr}

We begin by describing the simplest version of our construction, which gives a new proof of the existence of (many) functions of the kind considered in the introduction:

\begin{theorem}\label{thm:Construction}
Let $\{g_n(x)\}_{n \geq 0}$ be any countable family of functions $[0,1] \to [0,1]$. 
There exists a strictly increasing analytic function $f:[0,1] \to [0,1]$ such that
\begin{enumerate}
\item $f$ restricts to a bijection $\mathbb{Q} \cap [0,1] \to \mathbb{Q} \cap [0,1]$;
\item $f$ is different from all the $g_n(x)$.
\end{enumerate}
In particular, since the set of all rational functions with rational coefficients is countable, there exists such an analytic function that is not a rational function.
\end{theorem}

The idea is simple: we enumerate the rational numbers contained in the interval $[0,1]$ as $x_0, x_1, \ldots$, and construct a sequence of (strictly increasing) polynomials $f_n(x)$ such that $f_n(x_i)$ is rational for all $i \leq n/2$ and $x_i$ is in the image of $f_n$ for all $i \leq n/2$. We make this construction in such a way that $f_{n+1}(x_i)=f_n(x_i)$ for all $i=0,\ldots,\lfloor n/2 \rfloor$, which ensures that at least the first of these two properties is preserved in the passage to the limit. Moreover, we can also make the second property pass to the limit if we additionally require that (at least for $n$ large enough) the inverse image $f_n^{-1}(x_i)$ does not depend on $n$.

The proof we give below implements exactly this idea, up to a little bookkeeping to keep track of precisely \textit{which} rationals have already been considered.

\begin{proof}
Let $\{x_n\}_{n \geq 0}, \{y_n\}_{n \geq 0}$ be two (not necessarily distinct) enumerations of the rationals in $[0,1]$, with $x_0=y_0=0, x_1=y_1=1$.
We look for an $f$ of the form
\[
f(x) = \sum_{n=1}^\infty p_n(x)
\]
where the $p_n(x)$ (for $n \geq 1$) are polynomials satisfying the following properties:
\begin{itemize}
\item[(a)] $\sup_{z \in \mathbb{C}, |z| \leq 2 } |p_n(z)| \leq 4 \cdot (3/4)^{n}$ and $\sup_{x \in [0,1]} |p_n'(x)| \leq 4^{1-n}$;
\item[(b)] 
there is a bijective map
\[
\begin{array}{cccc}
 j : &  \mathbb{N} & \to & \mathbb{N} \\
& n & \mapsto & j_n
\end{array}
\]
such that $p_n(x_{j_m})=0$ for all $0 \leq m < n$. 
\item[(c)] $p_1(x)=x$ and $p_2(x)=0$.
\end{itemize}

Property (a) ensures that $f(x)$ is an analytic function on $[0,1]$: indeed if this property is satisfied then the series defining $f(x)$ converges uniformly on $D_{2}:=\{z \in \mathbb{C}, |z| < 2\}$, so $f(x)$ extends to a holomorphic function on all of $D_{2}$ and in particular it is real analytic on $[0,1]$. Properties (a) and (c) also ensure that $f$ is strictly increasing on the interval $[0,1]$, because
\begin{equation}\label{eq_lowboundderivative}
f'(x) = \sum_{n=1}^\infty p_n'(x) \geq p_1'(x) - \sum_{n \geq 2} |p_n'(x)| \geq 1 - \sum_{n \geq 1} 4^{-n} > 0.
\end{equation}

Notice that if the map $n \mapsto j_n$ is given, then in order to satisfy properties (a) and (b) one can simply take
\begin{equation}\label{eq_defPn}
p_n(x) = \frac{\varepsilon_n}{n} \prod_{k=0}^{n-1} (x-x_{j_k})
\end{equation}
if $\varepsilon_n$ is sufficiently small; we shall then make this choice from the start, namely, we set $p_n(x)$ to be the polynomial given by formula \eqref{eq_defPn}, and we shall choose $n \mapsto j_n$ and $\varepsilon_n$ in what follows. By the triangular inequality, for all $x \in D_2$ we have $\prod_{i=0}^{n-1} |x-x_i| \leq 3^n$, hence it is not hard to see that in order to satisfy the inequalities in (a) it suffices to take $\varepsilon_n$ in the interval $[0,4^{1-n}]$.

We shall write $f_n(x)$ for the partial sums $\sum_{m=1}^n p_m(x)$. Our choices imply that for all indices $n \in \mathbb{N}$ we have $f(x_{j_m})=f_n(x_{j_m})$ for all $n \geq m$, since
\[
f(x_m)-f_n(x_{j_m}) = \sum_{k>n} p_k(x_{j_m})
\] and the $p_k(x_{j_m})$ all vanish for $k>n \geq m$. 
Also notice that each function $x \mapsto p_k(x)$ is obviously continuous, and it is bijective from $[0,1]$ to itself: to see this, observe that properties (b) and (c) together with our definition of $p_n(x)$ imply $f_n(0)=0$ and $f_n(1)=1$ for all $n$, and furthermore by the same estimate as in equation \eqref{eq_lowboundderivative} we have $f_n'(x) > 0 $ for all $x \in (0,1)$. For later use, notice that we also have
\begin{equation}\label{eq:UpBoundDerivative}
|f_n'(x)| \leq 1 + \sum_{n \geq 1} 4^{-n} \leq 2 \quad \forall x \in (0,1) \text{ and } \forall n \geq 1.
\end{equation}

We shall now define the map $n \mapsto j_n$ and the parameters $\varepsilon_n$ recursively. We take $j_0=0$, $j_1=1$ and $\varepsilon_1=1$, $\varepsilon_2=0$, so that $f_2(x)=f_1(x)=p_1(x)=x$ and $p_2(x)=0$. Now assume that we have already defined $j_n$ and $\varepsilon_n$ for all $n \leq m$ and set $J_m:=\{j_0,\ldots,j_m\}$. We shall show that we can construct $\varepsilon_{m+1}$ and $j_{m+1}$, and that in fact one can take all the $\varepsilon_n$ to be rational numbers.
Together with the choices we have already made, this implies that our inductive construction satisfies the following properties: for all $j \in J_m$ and all $n \geq m$ we have $f_n(x_j)=f_m(x_j)=f(x_j)$; moreover, $f_n(x)$ is a polynomial with \textit{rational} coefficients, hence for all $j \in J_m$ we have $f(x_j)=f_m(x_j) \in \mathbb{Q}$.
We distinguish two cases:
\begin{enumerate}
\item Suppose that $m+1$ is odd. Let $a=\min (\mathbb{N} \setminus J_m)$ and set $j_{m+1}=a$. Notice that the set 
$\{j_0,j_1,\ldots,j_{m}\}$ has cardinality $m+1$, hence $a \leq m+1$.
We have
\[
f(x_{j_{m+1}})=f(x_a)=\sum_{n = 1}^\infty p_n(x_a) = \sum_{n \leq m} p_n(x_a) + p_{m+1}(x_a) + \sum_{n > m+1} p_n(x_a);
\]
independently of the choice of the parameters $\varepsilon_n$ for $n>m+1$, our construction ensures that  $p_n(x_a)=0$ for all $n>m+1 \geq a$, so we have $\sum_{n > m+1} p_n(x_a)=0$ and
\[
f(x_a)=f_{m+1}(x_a)= \sum_{n \leq m} p_n(x_a) + \frac{\varepsilon_{m+1}}{m+1} \prod_{k=0}^{m} (x_a-x_{j_k});
\]
what we require is that $f(x_a)$ be a rational number, and that $\varepsilon_{m+1}$ be sufficiently small and rational. Since rational numbers are dense in $\mathbb{R}$, it is clear that we can choose a rational number $z \in [0,1]$ that satisfies all of the following properties:
\begin{itemize}
\item $z \not \in \{f(x_j) \bigm\vert j \in J_m \}$;
\item the quantity
\[
 (m+1) \cdot  \left(z - \sum_{n \leq m} p_n(x_a)\right) \cdot \prod_{k=0}^{m} (x_a - x_{j_k})^{-1}
\]
does not exceed $4^{-m}$;
\item $z \neq g_{m/2-1}(x_a)$ (recall that $\{g_n\}_{n \in \mathbb{N}}$ is the given set of functions we wish to avoid).
\end{itemize}
We set $\displaystyle \varepsilon_{m+1}:= (m+1) \cdot  \left(z - \sum_{n \leq m} p_n(x_a)\right) \cdot \prod_{k=0}^{m} (x_a - x_{j_k})^{-1}$, which is easily seen to be a rational number. By construction, this choice ensures that
\[
f(x_a)=f_{m+1}(x_a)=z \in \mathbb{Q} \cap [0,1]
\]
and $|\varepsilon_{m+1}| \leq 4^{-m}$; furthermore, it also ensures that $f(x) \neq g_{m/2-1}(x)$ as functions, since we have $f(x_a)=z \neq g_{m/2-1}(x_a)$, and this independently of the choice of $\varepsilon_n$ for $n>m+1$.

\item Suppose that $m+1$ is even. Let $b$ the least natural number such that $y_b$ does not belong to the set $
\{f_m(x_{j}) \bigm\vert j \in J_m\}.
$
We want to choose $j_{m+1}$ and $\varepsilon_{m+1}$ in such a way that $f(x_{j_{m+1}})=y_b$ and $\varepsilon_{m+1}$ is again sufficiently small and rational. Consider the function
\[
h_{m} : x \mapsto \frac{y_b-f_m(x)}{\prod_{k=0}^{m} (x-x_{j_k})};
\]
as $f_m(x) : [0,1] \to [0,1]$ is a bijection, there exists a unique $\overline{x} \in [0,1]$ such that $f_m(\overline{x})=y_b$. By assumption, $\overline{x} \not \in \{ x_{j_0}, x_{j_1}, \ldots, x_{j_{m}}  \}$, so the function $h_{m}$ is continuous in a neighbourhood of $\overline{x}$, since the denominator does not vanish for $x$ sufficiently close to $\overline{x}$. Because of the density of $\mathbb{Q} \cap [0,1]$ in $[0,1]$ and of the fact that $h_{m}(x)$ is continuous in a neighbourhood of $\overline{x}$ and satisfies $h_{m}(\overline{x})=0$, there exists a rational number $z \in [0,1] \setminus \{x_{j_0}, x_{j_1},\ldots, x_{j_{m}}\}$ such that $|h_{m}(z)| < (m+1)^{-1} 4^{-m}$. We set $j_{m+1}$ to be the unique index such that $x_{j_{m+1}}=z$; the construction ensures that $j_{m+1} \not \in \{j_0,\ldots,j_{m}\}$. Finally, we take 
$
\varepsilon_{m+1}:= (m+1) h_{m}(z),
$
which by construction satisfies $|\varepsilon_{m+1}| < 4^{-m}$, and which is a rational number since $h_m(x)$ is a rational function (with rational coefficients). We then have
\[
f(x_{j_{m+1}}) = f_{m+1}(z)= \sum_{n \leq m+1} p_{n}(z) = \sum_{n \leq m} p_{n}(z) + \frac{m+1}{m+1} h_{m}(z) \cdot \prod_{k=0}^m (z-x_{j_k}) = y_b,
\]
and this independently of the choice of $\varepsilon_n$ for $n>m+1$.
\end{enumerate}
It is clear that we can carry out this construction for all $m$. 
We claim that the resulting function $f(x)$ satisfies the properties given in the statement. Indeed:
\begin{itemize}
\item step (1) of the above procedure ensures that $\min (\mathbb{N} \setminus J_{2k})$ is strictly increasing as a function of $k$, hence that $\bigcup_{m \geq 0} J_m=\mathbb{N}$. As we have already seen, $f(x_j)$ is rational whenever $j \in J_m$ for some $m$, hence $f(x_j) \in \mathbb{Q} \cap [0,1]$ for all $j \in \bigcup_{m \geq 0} J_m = \mathbb{N}$. Since $\{x_j \bigm\vert j \in \mathbb{N}\}=\mathbb{Q} \cap [0,1]$, this implies that $f(\mathbb{Q} \cap [0,1]) \subseteq \mathbb{Q} \cap [0,1]$.
\item when applying step (1) of the above procedure for $m=2k$, $k \geq 1$, we make certain that $f(x) \neq g_{k-1}(x)$. Since $k-1$ ranges over all the natural numbers, this implies that $f(x) \neq g_k(x)$ for all $k$.
\item finally, step (2) ensures that the quantity
\[
\min \left\{ b \bigm\vert \forall j \in J_{2k} \text{ we have } y_b \neq f_{2k}(x_j) \right\}
\]
is strictly increasing as a function of $k$, hence for all $b \in \mathbb{N}$ there exists an $m \in \mathbb{N}$ large enough that $y_b=f_m(x_j)$ for some $j \in J_m$. As we have already seen, this implies $f(x_j)=f_m(x_j)=y_b$, so $y_b \in f(\mathbb{Q} \cap [0,1])$. Since this holds for all $b$ and we have $\{y_b \bigm\vert b \in \mathbb{N}\} = \mathbb{Q} \cap [0,1]$ by construction, this implies that $f(\mathbb{Q} \cap [0,1])$ is onto $\mathbb{Q} \cap [0,1]$ as claimed.
\end{itemize}
\end{proof}

\begin{remark}
It is not hard to realize that, since we can choose a countable number of parameters $\varepsilon_n$, and for each we have countably many choices, the set of functions $f$ satisfying the conclusion of the theorem has the cardinality of the continuum. This gives a different (and perhaps more natural) proof of the fact that we can avoid any given set of functions, as long as it is countable. The presentation we have decided to give, on the other hand, has the advantage to make clear that the whole procedure is completely constructive.
\end{remark}

\begin{remark}\label{rmk:DifferentProof}
As it was already true of Franklin's method \cite{MR1501300}, a slight modification of the proof of theorem \ref{thm:Construction} shows the following stronger result: given $\varepsilon>0$ and a strictly increasing analytic function $g:[0,1] \to [0,1]$ with $g(0)=0$ and $g(1)=1$, there exists an analytic function $f:[0,1] \to [0,1]$ that induces a bijection of $\mathbb{Q} \cap [0,1]$ and such that $\|f-g\|_\infty < \varepsilon$, where $\| \cdot \|_{\infty}$ denotes the supremum (or uniform) norm. The modifications one needs to make to the previous argument are minimal: we simply start with $f_0(x)=g(x)$, and at each step we choose $\varepsilon_n$ to be a \textit{real} number smaller than $ 4^{-n}\varepsilon$.

This also gives a different proof of the existence of transcendental functions that induce bijections of $\mathbb{Q} \cap [0,1]$ with itself. Indeed, we know from lemma \ref{lemma:OnlyRationalFunctions} that the algebraic functions with this property are very sparse, so it's easy to see that we can choose a strictly increasing analytic function $g:[0,1] \to [0,1]$ which is far from all of them in the supremum norm and satisfies $g(0)=0$, $g(1)=1$. We then use the argument just sketched to produce an analytic function $f(x)$ inducing a bijection on $\mathbb{Q} \cap [0,1]$ and very close to $g(x)$ in the uniform norm: provided that $\|f-g\|_\infty$ is small enough, $f(x)$ cannot be any of the functions described in lemma \ref{lemma:OnlyRationalFunctions}, so it is a transcendental function with the property we are interested in.
The author is grateful to Umberto Zannier for this remark.
\end{remark}

\section{Height bounds}\label{sec:HeightBounds}
In the interest of clarity we now briefly discuss our conventions for the notion of height of a rational number.  For $x \in \mathbb{Q}$ we write $D(x) \in \mathbb{N}_{>0}$ (resp.~$N(x) \in \mathbb{Z}$) for the denominator (resp.~numerator) of $x$ when it is written in lowest terms. By the \textit{height} of $x$ we mean its \textit{logarithmic height}, namely
\[
h(x) = \log \max\{|N(x)|,D(x)\};
\]
we shall also use $H(x)$ to denote $\max\{|N(x)|,D(x)\}$. Notice that if $x$ is a rational in the interval $[0,1]$, then $h(x)=\log D(x)$.
The function $D$ obviously satisfies the following properties:
\[
D(x_1+ \ldots + x_n) \leq \operatorname{lcm} \{ D(x_1), \cdots,  D(x_n) \}, \quad D(x_1x_2)=D(x_1)D(x_2);
\]
analogously, the function $h$ satisfies (see for example \cite[Chapter 3]{MR1756786})
\[
h(x_1+\ldots+x_n) \leq h(x_1)+\ldots+h(x_n)+(n-1)\log 2, \; h(x_1x_2) \leq h(x_1)+h(x_2), \; h(1/x)=h(x).
\]
We shall make free use of these properties without further comment.

We can now define the \textit{lexicographic ordering} $\prec$ on the rational numbers in the interval $[0,1]$ as follows: we say that $q_1 \prec q_2$ if either $H(q_1) < H(q_2)$ holds, or we have both $H(q_1)=H(q_2)$ and $q_1 < q_2$. It is easy to see that this is a well-ordering of $\mathbb{Q} \cap [0,1]$. We can then define the \textit{lexicographic enumeration} $x_0, x_1, \ldots$ of the rationals in $[0,1]$: we set $x_0=0$ and, for $n \geq 0$,
\[
x_{n+1} = \min_{\prec} (\mathbb{Q} \cap [0,1]) \setminus \{x_0,\ldots,x_n\},
\]
where by $\min_{\prec}$ we mean the minimum with respect to the lexicographic ordering. It is easy to check that the following lemma holds.

\begin{lemma}\label{lemma:IndexGrowth}
Let $(x_n)_{n \in \mathbb{N}}$ be the lexicographic enumeration of the rationals in the interval $[0,1]$. For all $n \geq 2$ we have
$
H(x_n) \geq \sqrt{2n};
$
equivalently, given $q \in \mathbb{Q} \cap (0,1)$, the unique index $n$ for which $q=x_n$ satisfies $n \leq \frac{H(q)^2}{2}$.
\end{lemma}
\begin{remark}
Asymptotically, these inequalities are not sharp: indeed, it is well-known that $\lim_{n \to \infty} \frac{H(x_n)}{\sqrt{n}} =\frac{\pi}{\sqrt{3}}$. However, the only possible improvement lies in the constant factor sitting in front of $\sqrt{n}$ (resp.~of $H(q)^2$), and not in the functional form of the bound; since we are not interested in especially sharp results, we chose to use the inequalities of lemma \ref{lemma:IndexGrowth} because of their particularly simple form.
\end{remark}

We can now prove the following strengthening of theorem \ref{thm:Construction}:
\begin{theorem}
Let $\{g_n(x)\}_{n \geq 0}$ be any countable family of functions $[0,1] \to [0,1]$. 
There exists a strictly increasing analytic function $f:[0,1] \to [0,1]$ such that
\begin{enumerate}
\item $f$ restricts to a bijection $\mathbb{Q} \cap [0,1] \to \mathbb{Q} \cap [0,1]$;
\item $f$ is different from all the $g_n(x)$;
\item $h(f(x)) \leq B(H(x)^2)$, where $B: \mathbb{N} \setminus\{0\} \to \mathbb{N}$ is given by
$
B(t) = 4 t \cdot 48^{t} \cdot \Gamma(t).
$
\end{enumerate}

\end{theorem}

\begin{proof}
We follow closely the proof of theorem \ref{thm:Construction} (keeping in particular all the notation), and only point out the necessary adjustments to the argument.
Let
\[
X(n) = \begin{cases} 48^t \Gamma(t), \text{ if } t \geq 1 \\ 1, \text{ if } t=0,
\end{cases}
\]
and notice that the function $X(n)$ satisfies the inequality
\[
\sum_{k=0}^{n-1} X(k) \leq X(n) \quad \forall n \geq 1;
\]
we shall need this fact in what follows, and we will often use it in the equivalent form $\sum_{k=0}^{n} X(k) \leq 2X(n)$.
We take $x_n$ and $y_n$ to both be the lexicographic enumeration of the rationals; we shall endeavour to choose the sequences $j_n, \varepsilon_n$ in such a way that the following hold: 
\begin{enumerate}
\item $\displaystyle h \left( \frac{\varepsilon_n}{n} \right) \leq nX(n)$ for all $n \geq 1$;
\item $h(x_{j_k}) \leq X(n)$ for all $0 \leq k \leq n$.
\end{enumerate}

Assuming now that we can indeed choose $j_n, \varepsilon_n$ so as to satisfy 1 and 2 above, for all $x \in \mathbb{Q} \cap [0,1]$ and for all $n \geq 1$ we have the following inequalities:
\begin{equation}\label{eq:UpperBoundfn0}
\begin{aligned}
h(f_n(x)) & = \log D(f_n(x)) \\
& = \log D\left( \sum_{m=1}^n p_m(x) \right) \\
& \leq \log \operatorname{lcm} \{ D(p_m(x)) \bigm\vert m=1,\ldots,n \} \\
& = \log \operatorname{lcm} \left\{ D\left( \frac{\varepsilon_m}{m} \prod_{k=0}^{m-1} (x-x_{j_k})  \right) \bigm\vert m=1,\ldots,n \right\} \\
& \leq \sum_{m=1}^n \log D\left(\frac{\varepsilon_m}{m}\right) + \sum_{k=0}^{n-1}\log D(x-x_{j_k}) \\
& \leq \sum_{m=1}^n h\left(\frac{\varepsilon_m}{m}\right) + \sum_{k=0}^{n-1} (\log D(x)+\log D(x_{j_k}) )\\
& = \sum_{m=1}^n h\left(\frac{\varepsilon_m}{m}\right) + \sum_{k=0}^{n-1} (h(x)+h(x_{j_k}) )\\
& \leq \sum_{m=1}^n mX(m) + n h(x) + \sum_{k=0}^{n-1}X(k) \\
& \leq n \sum_{m=1}^n X(m) + n h(x) + \sum_{k=0}^{n-1}X(k) \\
& \leq n h(x) + 3nX(n).
\end{aligned}
\end{equation}
In particular, if we evaluate $f_n(x)$ at $x=x_{j_k}$ with $k \leq n$ we have
$
h(x_{j_k}) \leq X(k) \leq X(n),
$
hence $h(f_n(x_{j_k})) \leq 4nX(n)$; since furthermore we have $f(x_{j_k})=f_k(x_{j_k})$ for all $k \geq 1$, we obtain for all $k \geq 1$ the inequality
\begin{equation}\label{eq:UpperBoundfn}
h(f(x_{j_k}))=h(f_k(x_{j_k})) \leq 4kX(k).
\end{equation}
Furthermore, we claim that, given $x \in \mathbb{Q} \cap [0,1]$, the corresponding index $k$ such that $x=x_{j_k}$ satisfies $k \leq H(x)^2$. We now prove this statement. Notice first that this is obviously true for $k=0,1$, so we can assume $k \geq 2$. Following the procedure described in the proof of theorem \ref{thm:Construction}, at every step such that $m+1$ is odd we let $j_{m+1}=\min \mathbb{N} \setminus J_m$; as we have $j_0=0$ and $j_1=1$, this implies that, for all integers $t \geq 0$, at step $m+1=2t+1$ we have $j_{m+1} \geq t+1$, which means that all the $x_n$ with $n \leq t$ are among the $x_{j_s}$ for $s \leq 2t$. 
Hence, letting $t$ be the index such that $x=x_t$, the unique index $k$ such that $t=j_k$ satisfies
\[
k \leq 2t \leq H(x_t)^2 = H(x)^2,
\]
where we have used lemma \ref{lemma:IndexGrowth} (recall that we have assumed $k \geq 2$). 
From this fact and equation \eqref{eq:UpperBoundfn} we then deduce the inequality
\[
h(f(x)) \leq 4H(x)^2 X\left(H(x)^2\right).
\]

To establish the theorem, therefore, it suffices to show that it is possible to choose the sequences $j_n$ and $\varepsilon_n$ so as to satisfy conditions 1 and 2 above.
Again we consider separately the case of $m+1$ being odd or even. 
\begin{itemize}
\item \textbf{$m+1$ is odd.} We have $a \leq m+1$, hence \[h(x_{j_{m+1}}) = h(x_a) \leq h(x_{m+1}) \leq m+1 \leq X(m+1).
\]
We now need to choose $z$ and $\varepsilon_{m+1}$, which are related by 
\[
z=\sum_{n \leq m} p_n(x_a) + \frac{\varepsilon_{m+1}}{m+1} \prod_{k=0}^m (x_a-x_{j_k}),
\]
in such a way that $\varepsilon_{m+1}$ does not exceed $4^{-m}$ and the corresponding $z$ does not belong to the set
\[
\{f(x_j) \bigm\vert j \in J_m\} \cup \{g_{m/2-1}(x_a)\}.
\]
Since this set has cardinality $m+2$ and the map $\varepsilon_{m+1} \mapsto z$ is injective, there are at most $m+2$ values of $\varepsilon_{m+1}$ that we need to exclude. Hence there exists an $s \in \{0,\ldots,m+2\}$ such that 
$
\frac{s}{(m+2)4^{m}} \leq 4^{-m}
$
is an acceptable value of $\varepsilon_{m+1}$. Finally, for the heights of $x_{j_{m+1}}=x_a$ and $\frac{\varepsilon_{m+1}}{m+1}$ we have the estimates
\[
h(x_a) = \log D(x_a) \leq \log a \leq \log(m+1) \leq X(m+1)
\]
and
\[
h\left( \frac{\varepsilon_{m+1}}{m+1} \right)=\log D\left( \frac{\varepsilon_{m+1}}{m+1} \right) = \log\left((m+1)(m+2)\right) + m \log 4 \leq (m+1)X(m+1),
\]
which finishes the inductive step in this case.

\item \textbf{$m+1$ is even.} Notice first that we have $b \leq m+1$, hence $h(y_b) \leq \log(m+1)$. 
Recall then that we defined $\varepsilon_{m+1}$ by the formula
\begin{equation}\label{eq:epsilonNew}
\varepsilon_{m+1} = (m+1) \frac{y_b-f_m(z)}{\prod_{k=0}^m (z-{x_{j_k}})},
\end{equation}
where $z$ is a rational number sufficiently close to $\overline{x}$ (the only real number in $[0,1]$ such that $f_m(\overline{x})=y_b$). We now want to show that $z$ can be chosen to be of controlled height, and use this fact to also bound the height of $\varepsilon_{m+1}$.

Let $a_0 < a_1 < \ldots < a_m$ be the increasing reordering of the points $x_{j_0}, x_{j_1}, \ldots, x_{j_m}$. There is a unique index $t$ such that $a_t < \overline{x} < a_{t+1}$. For ease of exposition, let us assume that $a_{t+1}-\overline{x} \leq \overline{x} - a_t$ (that is, $\overline{x}$ lies to the right of the midpoint of the segment $[a_t,a_{t+1}]$), the other case being perfectly symmetric. We take $z$ to be the maximum of the set
\[
\{ q \in \mathbb{Q} \cap [0,1] \bigm\vert H(q) \leq M:=\D, q < \overline{x} \}.
\]
Notice that the distance between $z$ and $\overline{x}$ is at most $1/M$. We now estimate the corresponding value of $\varepsilon_{m+1}$, studying separately numerator and denominator of \eqref{eq:epsilonNew}.

As for the former, we have already remarked that the derivative of $f_n(x)$ is bounded in absolute value by $2$ (see equation \eqref{eq:UpBoundDerivative}); from Lagrange's theorem we then get
\[
|f_m(z)-y_b| = |f_m(z)-f_m(\overline{x})| = |f_m'(\xi)(z-\overline{x})| \leq 2|z-\overline{x}| \leq \frac{2}{M},
\]
where $\xi$ is a suitable point between $z$ and $\overline{x}$.

\smallskip

Now consider the denominator of the right hand side of \eqref{eq:epsilonNew}. Notice that for $k<t$ we have
\[
|z-a_k| = (z-a_{k+1})+(a_{k+1}-a_k) \geq a_{k+1}-a_k
\]
and since $a_{k+1}, a_k$ are distinct we have
\begin{equation}\label{eq:LowerBoundDifference}
\begin{aligned}
|z-a_k| & \geq a_{k+1}-a_k \\
&\geq \frac{1}{D(a_{k+1}-a_k)}\\
&  = \exp(-\log D(a_{k+1}-a_k)) \\
& \geq \exp( -\log D(a_{k+1}) -\log D(a_k) ) \\
& = \exp( -h(a_{k+1}) -h(a_k));
\end{aligned}
\end{equation}
a similar argument works for $k>t+1$. Recalling that the $a_i$, $i=0,\ldots,m$, are a permutation of the $x_{{j_k}}$, $k=0,\ldots,m$, and that by the inductive assumption we have $H(x_{j_k}) \leq X(k)$, we obtain
\begin{equation}\label{eq_UpBoundDenominator}
\begin{aligned}
& \left|\frac{1}{\prod_k (z-x_{j_k})}\right| \leq \\
& \leq \frac{1}{|(z-a_t)(z-a_{t+1})|} \cdot \prod_{k<t} \exp(h(a_{k+1})+h(a_k)) \cdot \prod_{k>t+1} \exp(h(a_{k-1})+h(a_k)) \\
& < \frac{1}{|(z-a_t)(z-a_{t+1})|} \cdot \exp\left(2 \sum_{k=0}^m X(k)\right) \\
& \leq \frac{1}{|(z-a_t)(z-a_{t+1})|} \cdot \exp\left(4X(m)\right).
\end{aligned}
\end{equation}
Thus we only need to estimate the distances $|z-a_t|, |z-a_{t+1}|$.

\begin{itemize}
\item If $|z-a_t|<|z-a_{t+1}|$, then $z$ lies to the left of the midpoint of the segment $[a_t,a_{t+1}]$ while $\overline{x}$ (by assumption) lies to the right of it; hence we have
\[
\frac{a_t+a_{t+1}}{2} - \frac{1}{M} \leq \overline{x}-\frac{1}{M} \leq z < \frac{a_t+a_{t+1}}{2}
\]
and therefore
\[
\frac{a_{t+1}-a_t}{2} -\frac{1}{M} \leq z-a_t < \frac{a_{t+1}-a_t}{2}.
\]
By the same argument as in equation \eqref{eq:LowerBoundDifference}, and using $h(a_t), h(a_{t+1}) \leq X(m)$, we then get
\[
\begin{aligned}
z-a_t & \geq \frac{a_{t+1}-a_t}{2} - \frac{1}{M} \\
& \geq \frac{1}{2}\exp(-2X(m)) - \frac{1}{M} \\
& \geq \frac{1}{3} \exp(-2X(m)).
\end{aligned}
\]
Since $a_{t+1}-z \geq z-a_t$ by assumption, we finally obtain
\[
\begin{aligned}
\left|\frac{1}{\prod_k (z-x_{j_k})}\right| & \leq \frac{\exp\left(4X(m) \right)}{|(z-a_t)(z-a_{t+1})|} \\ & \leq 9 \exp\left(4X(m) + 4X(m)\right) \\ & = 9 \exp(8X(m)).
\end{aligned}
\]
\item If $|z-a_t| \geq |z-a_{t+1}|$, then it suffices to give a lower bound for $|z-a_{t+1}|$, which we do as follows. By construction we have $z<\overline{x}<a_{t+1}$, so it suffices to give a lower bound for $a_{t+1}-\overline{x}$. We set $q_m(x):=y_b-f_m(x)$ and observe that using Lagrange's theorem we have
\begin{equation}\label{eq:lowBoundOverlineX}
\begin{aligned}
|q_m(a_{t+1})| & = |q_m(a_{t+1})-q_m(\overline{x})| \\
& = |q_m'(\xi)(a_{t+1}-\overline{x})| \\
& = |f_m'(\xi)| \cdot (a_{t+1}-\overline{x}) \\
& \leq 2(a_{t+1}-\overline{x}),
\end{aligned}
\end{equation}
where $\xi$ is a certain point in the interval $(\overline{x},a_{t+1})$. Thus it suffices to give a lower bound for $|q_m(a_{t+1})|=|f_m(a_{t+1})-y_b|$: notice that this number is nonzero (by assumption $y$ does not belong to the set $f_m(J_m)$) and its height is at most
\[
\log D(h(y_b)) + \log D(f_m(a_{t+1}))
= h(y_b) + h(f_m(a_{t+1})) \leq \log (m+1) + 4mX(m),
\]
where we have used inequality \eqref{eq:UpperBoundfn} and the fact that $a_{t+1}$ is one of the $m+1$ numbers $x_{j_0}, \ldots, x_{j_m}$.
Thus we have
\[
|q_m(a_{t+1})| \geq \exp \left(-\log (m+1) -4mX(m)\right) = \frac{\exp\left(-4mX(m)\right)}{m+1};
\]
we deduce from \eqref{eq:lowBoundOverlineX} that
\[
z-a_t \geq a_{t+1}-z \geq a_{t+1}-\overline{x} \geq \frac{1}{2(m+1)} \exp\left(-4mX(m)\right),
\]
and putting everything together we obtain
\[
\begin{aligned}
\displaystyle \left|\frac{1}{\prod_k (z-x_{j_k})}\right| & \leq \frac{\exp(4X(m))}{|(z-a_t)(z-a_{t+1})|} \\ & \leq 4(m+1)^2 \cdot \exp\left(8mX(m) + 4X(m)\right) \\ & \leq \exp\left(13m X(m)\right).
\end{aligned}
\]
\end{itemize}

Thus we see that in all cases the quantity $\displaystyle \left|\frac{1}{\prod_{k=0}^m (z-x_{j_k})}\right|$ is bounded above by $\exp\left(13m X(m))\right)$.
Combining our bounds on the numerator and denominator of the right hand side of \eqref{eq:epsilonNew}, we see that our choice of $z$ leads to a value of $\varepsilon_{m+1}$ that is bounded above by
\[
\displaystyle \varepsilon_{m+1} \leq (m+1)\left|\frac{y_b-f_m(z)}{\prod_{k=0}^m (z-x_{j_k})}\right| \leq  \frac{2(m+1) \exp\left(13m X(m)\right)}{M} = 4^{-m},
\]
and $x_{j_{m+1}}:=z$ has height at most
\[
\log(M)=\log\left(\D\right) \leq 14 m X(m) < X(m+1).
\]
Finally, the height of $\displaystyle \frac{\varepsilon_{m+1}}{m+1}$ (that is, the logarithm of its denominator) is at most
\[
\begin{aligned}
\log D&(y_b-f_m(z)) +  \log \left|N\left( \prod_{k=0}^m (z-x_{j_k}) \right)\right| \\
& \leq h(y_b)+  h(f_m(z)) + \sum_{k=0}^m h(z-x_{j_k})  \\ & \leq h(y_b) + mh(z)+3mX(m) + (m+1)\log 2 + (m+1)h(z) + \sum_{k=0}^m h(x_{j_k}) \\
& \leq \log(m+1) + (2m+1)h(z) + 3mX(m)+ (m+1)\log 2 + \sum_{k=0}^m X(k) \\
& \leq \log(m+1) + (2m+1)h(z) + (m+1)\log 2 +(3m+2)X(m) \\
& \leq \log(m+1) + 14(2m+1)m X(m) + (m+1)\log 2 + (3m+2)X(m) \\
& \leq 48m^2X(m)= m (48m X(m)) =m X(m+1) \\ & <(m+1)X(m+1),
\end{aligned}
\]
where we have used \eqref{eq:UpperBoundfn0} on the second line and $h(z) \leq 14mX(m)$ on the fifth.

\end{itemize}
This concludes the inductive step, and therefore the proof of the theorem.
\end{proof}

\begin{remark}
While there is certainly room to improve the bound $B(t)$ of the previous theorem (for example, the numerical constant 48 is far from optimal), without any new ideas it seems unlikely that one can do substantially better than $B(t)=\Gamma(t)$; let us rapidly go through the proof again to see why we cannot expect to beat this bound.
In order to get a lower bound for the denominator of \eqref{eq:epsilonNew}, we estimate the height of $q_m(a_{t+1})$; since $q_m(x)$ is a polynomial of degree $m$ and the height of $a_{t+1}$ could potentially be comparable with $X(m)$, at least if $a_{t+1}=x_{j_m}$, the bound we get for $h(q_m(a_{t+1}))$ will roughly be of size $mX(m)$. On the other hand, in order for the ratio defining $\varepsilon_{n+1}$ to be small enough, we need at the very least the numerator to be smaller than the denominator; since the lower bound for the denominator is no be better than $\exp(-m X(m))$, in the notation of the previous proof we will have to take $M$ at least of size $\exp(mX(m))$, which means that we cannot rule out $z=x_{j_{m+1}}$ being of height $\approx mX(m)$. Hence with the present method we don't expect to be able to do better than $X(m+1) \geq mX(m)$, that is, $X(m) \geq \Gamma(m)$.
\end{remark}

\section{Graphs with ``many'' rational points of small height}\label{sec:ManyRationalPoints}
Recall the following celebrated result of Pila, already referred to in the introduction:
\begin{theorem}{(\cite[Theorem 9]{MR1115117})}\label{thm_Pila}
Let $f:[0,1] \to [0,1]$ be a transcendental analytic function. For all $\varepsilon>0$, the function
\[
C_f(T)=\# \left\{ x \in \mathbb{Q} \cap [0,1] \bigm\vert H(x) \leq T, H(f(x)) \leq T \right\}
\]
satisfies $\lim_{T \to \infty} C_f(T) T^{-\varepsilon}=0$.
\end{theorem}

One can ask whether this theorem is optimal, that is, if the gauge functions $x^\varepsilon$ can be replaced by anything smaller. The answer has again been given by Pila, who has shown that theorem \ref{thm_Pila} is indeed sharp, in the following sense.
We say that a function $s : \mathbb{R} \to \mathbb{R}$ is \textbf{slowly increasing} if for all $\varepsilon>0$ we have $\lim_{x \to \infty} x^{-\varepsilon} s(x)=0$. Pila constructed \cite[§7.5]{MR2068319}, for any slowly increasing function $s$, an analytic function $f$ and an unbounded sequence of positive integers $T_n$ such that $C_f(T_n) \geq s(T_n)$, which shows that theorem \ref{thm_Pila} cannot be substantially improved.

Through a slight modification of the construction of section \ref{sec:Constr} we now show that theorem \ref{thm_Pila} is sharp (in the sense above) also if we restrict our attention to functions $f:[0,1] \to [0,1]$ that induce bijections of $\mathbb{Q} \cap [0,1]$ with itself:
\begin{theorem}
Let $s(x)$ be a slowly increasing function and $\{g_n(x)\}_{n \in \mathbb{N}}$ be any countable sequence of functions $[0,1] \to [0,1]$. There exists a strictly increasing analytic function $f:[0,1] \to [0,1]$ such that
\begin{itemize}
\item $f$ restricts to a bijection $\mathbb{Q} \cap [0,1] \to \mathbb{Q} \cap [0,1]$;
\item $f$ is different from all the $g_n(x)$;

\item for infinitely many values of $T \in \mathbb{R}_{>0}$ we have
\begin{equation}\label{eq_CountingFunction}
C_f(T)=\# \left\{ x \in \mathbb{Q} \cap [0,1] \bigm\vert H(x) \leq T, H(f(x)) \leq T \right\} \geq s(T).
\end{equation}
\end{itemize}
\end{theorem}

As before, the idea is to use an iterative construction. What changes with respect to the proof of theorem \ref{thm:Construction}, however, is that we take many more steps of type (1) than steps of type (2), as we now make precise:
\begin{proof}

Let $(x_i)_{i \in \mathbb{N}}$ and $(y_j)_{j \in \mathbb{N}}$ be two enumerations of the rationals in $[0,1]$. While $y_j$ can be arbitrary, we take $x_i$ to be given by the lexicographic ordering as in the previous section: we set $x_0=0$ and, by induction, we let $x_{i+1}$ to be the (lexicographic) minimum of the set $\mathbb{Q} \cap [0,1] \setminus \{x_0,\ldots,x_i\}$. 
Again we shall construct the function $f(x)$ as a limit of polynomials $f_n(x) \in \mathbb{Q}[x]$, where
\[
f_{n+1}(x)=f_n(x) + \varepsilon_n \prod_{q \in Q_n} (x-q)
\]
for some rational number $\varepsilon_n$ and some subset $Q_n$ of $\mathbb{Q} \cap [0,1]$. We shall require that $Q_n \subseteq Q_{n+1}$ for all $n$. We shall also construct an auxiliary sequence $z_n$ of rational numbers with the property that $z_n$ is the inverse image of $y_n$ through the limit function $f(x)$.
In the first step of the recursion we set $f_0(x)=x$. We now show how to construct $\varepsilon_n, Q_n$, and $z_n$ assuming that $f_n(x)$ has been defined.

Since $f_n(x)$ is a polynomial, say of degree $d_n$, we can find a constant $b_n$ large enough that for all rational numbers in $[0,1]$ we have
\[
H(f_n(x)) \leq b_n H(x)^{d_n}.
\]
We can assume without loss of generality that $b_n \geq 1$, and we obtain the existence of a constant $c_n>0$ such that for all $T \geq b_n$ we have
\[
\begin{aligned}
\# \{ q \in \mathbb{Q} \cap [0,1] & \bigm\vert H(q)\leq T, H(f_n(q)) \leq T \} \\
&  \geq \# \{ q \in \mathbb{Q} \cap [0,1] \bigm\vert H(q)\leq T, b_n H(q)^{d_n} \leq T \} \\
& = \# \{ q \in \mathbb{Q} \cap [0,1] \bigm\vert  H(q) \leq (T/b_n)^{1/d_n} \} \\
& \geq c_n T^{2/d_n}.
\end{aligned}
\]
Since by assumption $T^{-2/d_n}s(T)$ tends to 0 as $T$ tends to infinity, we can choose a value $T_n \in \mathbb{N}$, $T_n>b_n$, so large that $c_nT_n^{2/d_n} \geq s(T_n)$. Without loss of generality we shall also assume that the inequality $T_{n} \geq T_{n-1}+n$ holds, so that in particular the sequence $T_n$ satisfies $\lim_{n \to \infty} T_n = + \infty$. 

We now turn to the definition of the quantities $\varepsilon_n, Q_n, z_n$.
We start by setting
\[
Q_n = \{q \in \mathbb{Q} \cap [0,1] \bigm\vert H(q) \leq T_n\} \cup \{z_0,\ldots,z_{n-1}\},
\]
which obviously contains $Q_{n'}$ for all $n'<n$. Notice that for $n=0$ we assume $\{z_0,\ldots,z_{n-1}\}$ to be the empty set.
Independently of the choice of $\varepsilon_n$ or of any of the $Q_{n''}$ for $n''>n$ (as long as they contain $Q_n$), our choice of $Q_n$ implies that $f(q)=f_n(q)$ for all rationals $q$ of height at most $T_n$; in turn, this gives
\[
\begin{aligned}
\# \{ q \in \mathbb{Q} & \cap [0,1] \bigm\vert H(q) \leq T_n, H(f(q)) \leq T_n \} \\ 
& = \# \{ q \in \mathbb{Q} \cap [0,1] \bigm\vert H(q) \leq T_n, H(f_n(q)) \leq T_n \} > s(T_n),
\end{aligned}
\]
so that our limit function $f(x)$, if it exists, does indeed satisfy inequality \eqref{eq_CountingFunction} for infinitely many values of $T$.

We still need to ensure that our limit function $f(x)$ exists, is analytic, strictly increasing, induces a bijection from $\mathbb{Q} \cap [0,1]$ to itself, and is different from all the functions $g_n(x)$. This is done in the same spirit as in the proof of theorem \ref{thm:Construction}. More precisely,
\begin{enumerate}
\item by the same argument as in the proof of theorem \ref{thm:Construction}, in order to guarantee that the limit function $f(x)$ is analytic and monotonically increasing it suffices to choose the $\varepsilon_n$ to be sufficiently small (say less than $4^{-|Q_n|-1}|Q_n|^{-1}$);
\item the construction automatically implies that $f(x)$ is rational whenever $x$ is rational: indeed, for any given $x \in \mathbb{Q} \cap [0,1]$ there exists $n \in \mathbb{N}$ such that $T_n>H(x)$; it follows that $x$ belongs to $Q_m$ for all $m \geq n$, hence that $f(x)=f_{n}(x)$ is rational, because $f_n(x)$ is a polynomial with rational coefficients;
\item to ensure that the limit function $f(x)$ maps $\mathbb{Q} \cap [0,1]$ onto itself it suffices to ensure that every $y_j$ lies in the image of $f(x)$. This will be achieved by choosing the sequences $z_n, \varepsilon_n$ in such a way that $f_m(z_n)=y_n$ for all $m > n$;
\item we shall inductively choose the sequences $z_n, \varepsilon_n$ so as to ensure that $f(x)$ is distinct from all the $g_n(x)$; more precisely, at step $n$ of the construction we shall make sure that $f(x)$ is different from $g_n(x)$.
\end{enumerate}

Before proving that we can realize the construction in such a way as to satisfy constraints 1, 3 and 4 above, we make a preliminary remark.
Since the limit function $f(x)$ we are constructing is going to be a strictly increasing bijection of $[0,1]$ with itself, it will certainly be different from all the functions $g_n(x)$ that do not possess this property. Hence, replacing $(g_n)_{n \in \mathbb{N}}$ with a subsequence if necessary, we can assume that every $g_n(x)$ is a strictly increasing bijection of $[0,1]$ with itself: this slightly simplifies the argument to follow.

We now show that we can indeed achieve 1, 3 and 4.
Our construction of the sets $Q_n$ immediately implies that $f_m(z_n)=f_{n+1}(z_n)$ for all $m > n$, so in order for property 3 to be satisfied it suffices to choose $\varepsilon_n, z_n$ in such a way that
\[
y_n = f_{n+1}(z_n) = f_n(z_n) + \varepsilon_n \prod_{q \in Q_n} (z_n-q).
\]
We use this equation to define $\varepsilon_n$ in terms of $z_n$, so that we only need to choose the latter. Two cases arise:
\begin{itemize}
\item Suppose that we have $f_n(z)=y_n$ for some $z \in Q_n$. Then we have $f(z)=f_n(z)=y_n$, so in order to satisfy 3 we can simply take $z_n=z$, and in order to satisfy 1 it suffices to take $\varepsilon_n$ to be rational and smaller than $4^{-|Q_n|-1}|Q_n|^{-1}$. Hence we just need to prove that, with a suitable choice of $\varepsilon_n$, we can also make sure that 4 is satisfied.
To this end, consider the set
\[
\tilde{Q}_{n+1}=\{r \in \mathbb{Q} \cap [0,1] \bigm\vert T_n < H(r) \leq T_{n+1}\} \setminus \{z_0,\ldots, z_{n-1}\};
\]
we claim that it is nonempty. Indeed we have assumed $T_{n+1}$ to be at least $T_n+n+1$, so the cardinality of $\tilde{Q}_{n+1}$ is at least
\[
\begin{aligned}
\left|\{q \in \mathbb{Q} \cap [0,1] \right.&\left. \bigm\vert T_n < H(q) \leq T_{n+1}\}\right| - n \\ & \geq \left|\{q \in \mathbb{Q} \cap [0,1] \bigm\vert H(q) \in \{T_n+1,\ldots, T_n+n+1\} \}\right| - n \\
& \geq \left|\left\{ \frac{1}{T_n+1}, \ldots, \frac{1}{T_n+n+1} \right\}\right| -n = 1.
\end{aligned}
\] 
Let $r$ be any element of $\tilde{Q}_{n+1}$. Since $H(r) \leq T_{n+1}$, the number $r$ belongs to $Q_{m}$ for all $m \geq n+1$, hence
\[
f(r)=f_{n+1}(r) = f_n(r) + \varepsilon_n \prod_{q \in Q_n} (r-q);
\]
in order to make sure that $f(x) \neq g_n(x)$, it suffices to choose $\varepsilon_n$ in such a way that the above expression is different from $g_n(r)$.

\item If instead $y_n$ does not belong to the set $\{f_n(z) \bigm\vert z \in Q_n\}$, then, since $f_n$ is a bijection from $[0,1]$ to itself (cf.~the proof of theorem \ref{thm:Construction}), there is a $\overline{z} \in [0,1] \setminus Q_n$ such that $f_n(\overline{z})=y_n$. 
If we now choose $z_n$ to be close enough to $\overline{z}$, then (by continuity, and since the denominator does not vanish in $\overline{z} \not \in Q_n$) we can ensure that 
\[
\varepsilon_n = \frac{y_n-f_n(z_n)}{\prod_{q \in Q_n} (z_n-q)}
\]
is smaller than $4^{-|Q_n|-1}|Q_n|^{-1}$. Finally we can also make sure that $f(x) \neq g_n(x)$ by picking $z_n$ distinct from $g_n^{-1}(y_n)$; notice that this last condition makes sense, because $g_n:[0,1] \to [0,1]$ is a bijection, hence $g_n^{-1}(y_n)$ consists of precisely one point.
\end{itemize}

This concludes the iterative step of the construction, and shows that we can indeed find a function $f(x)$ as in the statement of the theorem.
\end{proof}

\noindent \textbf{Acknowledgements.} I am grateful to Umberto Zannier for bringing the problem to my attention and for his interest in this work. I also thank Marcello Mamino for useful discussions, and I acknowledge financial support from the Fondation Mathématique Jacques Hadamard (grant ANR-10-CAMP-0151-02).

\bibliographystyle{alpha}
\bibliography{Biblio}

\end{document}